\documentclass[12pt,reqno]{amsart}

\usepackage{graphics,cancel,graphicx}
\usepackage{epsfig,psfrag,manfnt}
\usepackage{bbm,subfigure,rotating,bm,float,mathdots,wasysym,scalerel}

\usepackage[all,cmtip]{xy}
\usepackage{amssymb,amsmath,amsfonts,amsthm,color}

\newenvironment{smallarray}[1]
 {\null\,\vcenter\bgroup\scriptsize
  \arraycolsep=.13885em
  \hbox\bgroup$\array{@{}#1@{}}}
 {\endarray$\egroup\egroup\,\null}
\usepackage{setspace}
\usepackage[new]{old-arrows}
\usepackage{stackengine}
\stackMath
\newcommand\reallywidehat[1]{%
\savestack{\tmpbox}{\stretchto{%
  \scaleto{%
    \scalerel*[\widthof{\ensuremath{#1}}]{\kern-.6pt\bigwedge\kern-.6pt}%
    {\rule[-\textheight/2]{1ex}{\textheight}}
  }{\textheight}%
}{0.5ex}}%
\stackon[1pt]{#1}{\tmpbox}%
}
\newcommand{\calA}{\mathcal{A}}

\newcommand{\calO}{\mathcal{O}}

\newcommand{\mC}{\mathbb{C}}
\newcommand{\mD}{\mathbb{D}}

\newcommand{\mF}{\mathbb{F}}

\newcommand{\mN}{\mathbb{N}}

\newcommand{\mR}{\mathbb{R}}

\newcommand{\mZ}{\mathbb{Z}}

\newcommand{\bbe}{\bm{e}}

\newcommand{\bbk}{\bm{k}}

\newcommand{\bbp}{\bm{p}}

\newcommand{\bbv}{\bm{v}}

\newcommand{\SL}{{\textrm{SL}}}
\newcommand{\eSL}{{\textrm{\em SL}}}
\newcommand{\E}{{\textrm{E}}}
\newcommand{\eE}{{\textrm{\em E}}}
\newcommand{\GL}{{\textrm{GL}}}
\newcommand{\eGL}{{\textrm{\em GL}}}

\newtheorem{theorem}{Theorem}[section]
\newtheorem{lemma}[theorem]{Lemma}
\newtheorem{corollary}[theorem]{Corollary}
\newtheorem{proposition}[theorem]{Proposition}

\theoremstyle{definition}

\newcommand{\nm}{\,\rule[-.6ex]{.13em}{2.3ex}\,}
\theoremstyle{definition}
\newtheorem{definition}[theorem]{Definition}

\theoremstyle{definition}

\theoremstyle{definition}
\newtheorem{example}[theorem]{Example}

\begin{document}



 \title[]{On a Banach algebra of entire functions with a weighted Hadamard multiplication}
 
 \author[]{Amol Sasane}
 \address{Department of Mathematics \\London School of Economics\\
     Houghton Street\\ London WC2A 2AE\\ United Kingdom}
 \email{A.J.Sasane@lse.ac.uk}
 
 \maketitle
 
 \begin{abstract} 
 New algebraic-analytic properties of a previously studied Banach algebra $\calA(\bbp)$ of entire functions are established. For a given fixed sequence $(\bbp(n))_{n\geq 0}$ of positive real numbers, such that $\lim_{n\rightarrow \infty} \bbp(n)^{\frac{1}{n}}=\infty$, the Banach algebra $\calA(\bbp)$ is the set of all entire functions $f$ such that $f(z)=\sum_{n=0}^\infty \widehat{f}(n) z^n $ ($z\in \mC$), where the sequence $(\widehat{f}(n))_{n\geq 0}$ of Taylor coefficients of $f$ satisfies $\widehat{f}(n)=O(\bbp(n)^{-1})$ for $n\rightarrow \infty$, with pointwise addition and scalar multiplication, a weighted Hadamard multiplication $\ast$ with weight given by $\bbp$ (i.e., $(f\ast g)(z)=\sum_{n=0}^\infty \bbp(n) \widehat{f}(n)\widehat{g}(n)z^n$ for all $z\in \mC$), and the norm $\|f\|=\sup_{n\geq 0} \bbp(n)|\widehat{f}(n)|$. The following results are shown: 
 \begin{itemize}
 \item The Topological stable rank of $\calA(\bbp)$ is $1$. 
 \item The Bass stable rank of $\calA(\bbp)$ is $1$. 
 \item $\calA(\bbp)$ is a Hermite ring. 
 \item $\calA(\bbp)$ is not a projective-free ring. 
 \item Idempotents in $\calA(\bbp)$ are described. 
 \item Exponentials in $\calA(\bbp)$ are described, and it is shown that every invertible element of $\calA(\bbp)$ has a logarithm, so that the first \v{C}ech cohomology group $H^1(M(\calA(\bbp)),\mZ)$ with integer coefficients of the maximal ideal space $M(\calA(\bbp))$ is trivial. 
 \item A generalised necessary and sufficient `corona-type condition' on the matricial
data $(A,b)$ with entries from $\calA(\bbp)$ is given for the solvability of $Ax = b$ with $x$ also having entries from $\calA(\bbp)$.
 \item The Krull dimension of $\calA(\bbp)$ is infinite. 
 \item $\calA(\bbp)$ is neither Artinian nor Noetherian.
 \item $\calA(\bbp)$ is coherent.
 \item The special linear group over $\calA(\bbp)$ is generated by elementary matrices. 
 \end{itemize}
 \end{abstract}
 
 \section{Introduction}
 
\noindent  In \cite{vRen}, the following Banach algebras were introduced. 
Throughout the article, we will use the notation $\mN$ for the set $\{1,2,3,\cdots\}$ of natural numbers, and $\mN_0:=\mN\cup \{0\}$.
 
 \begin{definition}[The Banach algebra $\calA(\bbp)$]$\;$
 \phantom{Let $\bbp:\mN_0\rightarrow (0,\infty)$ 
 be }
Let $\bbp:\mN_0\rightarrow (0,\infty)$ 
 be such that 
 $$
 \displaystyle 
 \lim_{n\rightarrow \infty} (\bbp(n))^{\frac{1}{n}}=\infty.
 $$ 
 Define 
 $$
 \calA(\bbp)=
 \Big\{ f\!:\!\mC\!\rightarrow\! \mC\;\!\Big|\;\! f(z)\!=\sum_{n=0}^\infty a_n z^n\;
 (z\!\in\! \mC), \;
 a_n\!=\!O\Big(\frac{1}{\bbp(n)}\Big) \textrm{ for } n\!\rightarrow \!\infty\Big\}.
 $$
 The $O$-notation here means as usual that there exists a constant $C>0$ such that $\bbp(n)|a_n|<C$ for all $n\in \mN_0$. 
 For $f\in \calA(\bbp)$, we set 
 $$
 \widehat{f}(n):=\frac{1}{n!}\frac{d^nf}{dz^n}(0)\quad (n\in \mN_0).
 $$
 With pointwise addition and scalar multiplication, $\calA(\bbp)$ is a complex vector space. 
 We equip $\calA(\bbp)$ with the weighted Hadamard multiplication $\ast$, given by 
 $$
( f\ast g)(z)=\sum_{n=0}^\infty  \bbp(n) \widehat{f}(n) \widehat{g}(n) z^n \;\;(z\in \mC), 
\textrm{ for all }f,g\in \calA(\bbp),
$$
and the norm $\|\cdot\|$, defined by 
$$
\|f\|=\sup_{n\in \mN_0} \bbp(n) |\widehat{f}(n)|  \textrm{ for all }f\in \calA(\bbp).
$$
 \end{definition}
 
 \noindent Then $\calA(\bbp)$ is a complex commutative unital Banach algebra with the unit element $\varepsilon$ given by 
\begin{equation}
\label{identity_element}
\varepsilon(z)={\scaleobj{0.93}{\sum_{n=0}^\infty}} \;\!\frac{1}{\bbp(n)} z^n\quad (z\in \mC).
\end{equation}
For example, if ${\bbp}_{\ast}(n)=n!$ ($n\in \mN_0$), then $\lim_{n\rightarrow\infty} (n!)^{\frac{1}{n}}=\infty$, and the corresponding Banach algebra $\calA({\bbp}_{\ast})$ has the identity element $\exp z$. This Banach algebra $\calA({\bbp}_{\ast})$ was introduced and studied in \cite{Sen}. 
 In \cite{vRen},  the more general Banach algebras $\calA(\bbp)$ were introduced, their ideal structure was studied, and the following results were shown. \begin{itemize}
 \item[(R1)]
 \label{3320231845} $g\in \calA(\bbp)$ is a divisor of $f\in \calA(\bbp)$ if and only if there exists a constant $C>0$ such that 
 $|\widehat{f}(n)|\leq C|\widehat{g}(n)|$ for all $n\in \mN_0$. 
 
 \noindent In particular, $f$ is invertible in $\calA(\bbp)$ if and only if there exists a $\delta>0$ such that  
  $
 |\widehat{f}(n)|\geq \frac{\delta}{\bbp(n)}$ for all $n\in \mN_0$.
  \item[(R2)] Every finite collection of functions $f_1,\cdots, f_K\in \calA(\bbp)$ ($K\in \mN$) has a greatest 
 common divisor $d\in\calA(\bbp)$. Up to invertible elements, $d$ is given by 
  $
 \widehat{d}(n)=\max_{1\leq k\leq K} |\widehat{f_k}(n)|$ for all $n\in \mN_0$.
 \item[(R3)] \label{5_3_2023_1335}For $f,f_1,\cdots, f_K\in \calA(\bbp)$, $f$ belongs to the ideal $\langle f_1,\cdots, f_K\rangle 
 $ in $\calA(\bbp)$ generated by $f_1,\cdots, f_K$ if and only if there exists a constant $C\!>\!0$ such that 
  $
| \widehat{f}(n)|\!\leq\! C \!\sum_{k=1}^K |\widehat{f_k}(n)|$ for all $n\!\in\! \mN_0$.
 \item[(R4)] \label{3320231651} Every finitely generated ideal in $\calA(\bbp)$ is principal. 
 \item[(R5)] \label{11_3_2023_1509}An ideal $I$ of $\calA(\bbp)$ is  {\em fixed} if there exists an $m\in \mN_0$ such that for all $f\in I$, $\widehat{f}(m)=0$. Then $I_m:=\{f\in \calA(\bbp): \widehat{f}(m)=0\}$ is a fixed, maximal ideal of $\calA(\bbp)$. Every fixed, maximal ideal of $\calA(\bbp)$ is $I_m$ for some $m\in \mN_0$. 
 \end{itemize}
 
 \vspace{-0.3cm}
 
 \noindent In this article, we will show the following:
 
  \vspace{-0.3cm}
 
 \begin{enumerate}
 \item The topological stable rank of $\calA(\bbp)$ is equal to $1$.
 \item The Bass stable rank of $\calA(\bbp)$ is equal to $1$. 
 \item $\calA(\bbp)$ is a Hermite ring.
 \item $\calA(\bbp)$ is not a projective-free ring.
  \item Idempotents in $\calA(\bbp)$ are described. 
  \item Exponentials in $\calA(\bbp)$ are described, and it is shown that every invertible element of $\calA(\bbp)$ has a logarithm, so that the first \v{C}ech cohomology group $H^1(M(\calA(\bbp)),\mZ)$ with integer coefficients of the maximal ideal space $M(\calA(\bbp))$ is trivial. 
  \item A generalised necessary and sufficient `corona-type condition' on the matricial
data $(A,b)$ with entries from $\calA(\bbp)$ is given for the solvability of $Ax = b$ with $x$ also having entries from $\calA(\bbp)$.
 \item The Krull dimension of $\calA(\bbp)$ is infinite. 
 \item $\calA(\bbp)$ is neither Artinian nor Noetherian.
 \item $\calA(\bbp)$ is coherent.
 \item The special linear group over $\calA(\bbp)$ is generated by elementary matrices.
 \end{enumerate}
 
 \vspace{-0.3cm}
 
 \noindent 
 The outline of this article is as follows:  In each section, we will first give the background of the property, by recalling key definitions, and then prove the property, possibly with additional commentary.

 \section{Topological stable rank of $\calA(\bbp)$ is $1$} 
 
 \noindent The notion of the topological stable rank for topological algebras was introduced by Rieffel in \cite{Rie} analogous to the $K$-theoretic concept of (Bass) stable rank, as well as several related numerical invariants.
  
\begin{definition}
Let $R$  be a commutative unital ring with identity element $1$. 
 We assume that $1\neq 0$, that is, $R$ is not the trivial ring $\{0\}$. 
 For $n\in \mN$, an $n$-tuple $(a_1,\cdots ,a_n) \in  R^n=R\times \cdots\times R$ ($n$ times) is said to be {\em invertible} 
(or {\em unimodular}), if there exists $(b_1 , \cdots , b_n) \in R^n$ such that the B\'ezout equation
$$
b_1a_1+\cdots +b_na_n =1
$$ 
is satisfied. The set of all invertible $n$-tuples is denoted by $U_n(R)$.

Now suppose that $A$ is a commutative unital Banach algebra. 
The {\em topological stable rank} of $A$ is the minimum $n\in \mN$ such that $U_n(A)$ is dense in $A^n$,  and it is infinite if no such $n$ exists.
\end{definition}

\noindent Here, $A^{n}$  is the normed space with the `Euclidean norm' $\|\cdot\|_2$ given by  
\begin{equation}
\label{2_3_2023_1449}
 \|\bbv\|_2^2 :=\|v_1\|^2+\cdots+\|v_n\|^2
 \end{equation}
 for all $\bbv$ in $A^{n}$, where 
  $\bbv$ has the components $v_1,\dots,v_n\in A$.
 
 \begin{theorem}
 \label{theorem_tsr_Ap}
 The topological stable rank of $\calA(\bbp)$ is $1$.
 \end{theorem}
 \begin{proof} Let $f\in \calA(\bbp)$, and $\epsilon>0$. 
Define $g$ by 
$$
\widehat{g}(n)=\left\{\begin{array}{cl} 
\widehat{f}(n) & \textrm{if } \bbp(n) |\widehat{f}(n)|>\epsilon ,\\[0.1cm] 
\frac{\epsilon}{\bbp(n)} & \textrm{if } \bbp(n) |\widehat{f}(n)|\leq \epsilon. 
\end{array}\right.
$$
Then 
$$
|\widehat{g}(n)-\widehat{f}(n)|
\left\{\begin{array}{ll} 
= 0 & \textrm{if } \bbp(n) |\widehat{f}(n)|>\epsilon \\[0.1cm]
\leq \frac{2\epsilon}{\bbp(n)} & \textrm{if } \bbp(n) |\widehat{f}(n)|\leq \epsilon
\end{array}\right.
$$
and so $g\in \calA(\bbp)$ and $\|g-f\|\leq 2\epsilon$. Also, 
$$
|\widehat{g}(n)|=\left\{\begin{array}{cl} 
|\widehat{f}(n)| & \textrm{if } \bbp(n) |\widehat{f}(n)|>\epsilon \\[0.1cm]
\frac{\epsilon}{\bbp(n)} & \textrm{if } \bbp(n) |\widehat{f}(n)|\leq \epsilon. 
\end{array}\right\}\geq \frac{\epsilon}{\bbp(n)} \quad (n\in \mN_0).
$$
Consequently, $g$ is invertible in $\calA(\bbp)$ by (R1) on page~\pageref{3320231845}. 
 \end{proof}

\section{Bass stable rank of $\calA(\bbp)$ is $1$}
 
 \noindent The notion of the Bass stable rank 
was introduced in algebraic $K$-theory by Bass \cite{Bas}, but it has since proved  useful in other branches of mathematics as well, for example, in analysis (see e.g. \cite{MorRup}). 
 
 \begin{definition}
 Let $R$ be a commutative unital ring with identity $1$. 
 Let $n\in \mN$. An $(n + 1)$-tuple $(a_1,\cdots ,a_{n+1},\alpha) \in  U_{n+1}(R)$ is  said to be {\em reducible} if
there exists an $n$-tuple $(h_1,\cdots, h_n) \in  R^n$ such that
$(a_1 + h_1\alpha ,\cdots ,a_n + h_n\alpha) \in  U_n(R)$.  The {\em Bass stable rank of $R$} is the smallest integer $n$ such that every element in $U_{n+1}(R)$ is reducible. It is infinite if no such $n$ exists.
 \end{definition}
 
 \vspace{-0.3cm}
 
 \noindent For a commutative unital Banach algebra, the Bass stable rank is at most equal to the 
 topological stable rank \cite[Theorem~2.3]{Rie}. By Theorem~\ref{theorem_tsr_Ap}, it follows immediately 
 that the Bass stable rank of $\calA(\bbp)$ is $1$. Nevertheless, we include a direct proof below. 
 
 \begin{theorem}
 \label{theorem_bsr_Ap}
 The Bass stable rank of $\calA(\bbp)$ is $1$.
 \end{theorem} 
 
 \vspace{-0.45cm}
 
 \begin{proof} Let $f_1,f_2,g_1,g_2\in \calA(\bbp)$ be such that 
 $g_1 \ast f_1+g_2 \ast f_2=\varepsilon$.  
 Define $u$ by 
  $$
  \begin{array}{rcl}
 \widehat{u}(n)\!\!\!&=&\!\!\!\frac{1}{\bbp(n)}+|\widehat{f_1}(n)|\quad (n\in \mN_0).
 \end{array}
 $$ 
 Then $|\widehat{u}(n)|=O(\frac{1}{\bbp(n)})$ as $n\rightarrow \infty$, and so $u\in \calA(\bbp)$.  Moreover, we have 
 $|\widehat{u}(n)|\geq \frac{1}{\bbp(n)}$ ($n\in \mN_0$), 
  and so $u$ is invertible in $\calA(\bbp)$ by (R1) (p.~\pageref{3320231845}). 
  Then $\widehat{u^{-1}}(n)=\frac{\widehat{u}(n)}{(\bbp(n))^2}$ ($n\in \mN_0$). Define $F_1=f_1 \ast u^{-1}\in \calA(\bbp)$. Thus
  $$
   \begin{array}{rcl}
  \widehat{F_1}(n)
  \!\!\!&=&\!\!\!\bbp(n)  \widehat{f_1}(n)  \frac{1}{(\bbp(n))^2}\frac{1}{\frac{1}{\bbp(n)}+|\widehat{f_1}(n)|} 
  =\frac{1}{\bbp(n)} \frac{\bbp(n)  \widehat{f_1}(n)}{(1+\bbp(n)| \widehat{f_1}(n)|)}.
  \end{array}
$$
We have $\|F\|\leq 1$ because the above yields $
|\widehat{F_1}(n)|\leq \frac{1}{\bbp(n)}\cdot 1$ ($n\in \mN_0$). 
 Define $G_1=g_1\ast u\in \calA(\bbp)$. Then 
$G_1 \ast F_1=g_1\ast u\ast f_1\ast u^{-1}=g_1 \ast f_1$, and so 
 $
\varepsilon= g_1 \ast f_1+g_2 \ast f_2=G_1  \ast F_1+ g_2 \ast f_2.
$ 
For $\epsilon:=\frac{1}{4}>0$, define $H_1$ by 
$$
\widehat{H_1}(n)=\left\{ \begin{array}{cl} 
 \frac{\epsilon}{\bbp(n)} &\textrm{if } |\widehat{G_1}(n)|\leq \frac{\epsilon}{\bbp(n)},\\[0.1cm]
\widehat{G_1}(n)  &\textrm{if } |\widehat{G_1}(n)|> \frac{\epsilon}{\bbp(n)}.
\end{array}\right.
$$
Then $H_1\in \calA(\bbp)$ since 
$$
\begin{array}{rcl}
|\widehat{H_1}(n)|
\!\!\!&\leq&\!\!\! \max\{ \frac{\epsilon}{\bbp(n)}, |\widehat{G_1}(n)| \} 
\leq  \max\{ \frac{\epsilon}{\bbp(n)},\frac{\|G_1\|}{\bbp(n)}\} 
\leq \frac{\|G_1\|+\epsilon}{\bbp(n)}=:\frac{C}{\bbp(n)}.
\end{array}
$$
The element $H_1$ `approximates' $G_1$, since for all $n\in \mN_0$, we have 
$|\widehat{H_1}(n)-\widehat{G_1}(n)|\leq \frac{2\epsilon}{\bbp(n)}$, 
so that $\| H_1-G_1\|\leq 2\epsilon=\frac{1}{2}$. Furthermore, $H_1$ is invertible by (R1), since $
|\widehat{H_1}(n)|\geq \frac{\epsilon}{\bbp(n)}$ ($n\in \mN_0$). 
 We have that 
 $
H_1 \ast F_1 +g_2 \ast f_2=H_1 \ast F_1 +(\varepsilon -G_1  \ast F_1)=\varepsilon +(H_1-G_1)\ast F_1,
$ 
which is invertible in $\calA(\bbp)$, as 
 $
\|(H_1-G_1)\ast F_1\|\leq \|H_1-G_1\|\|F_1\|\leq \frac{1}{2}\cdot 1<1
$ 
(see e.g. \cite[Thm. 10.7]{Rud}). 
Setting $h=H_1^{-1}\ast u\ast g_2$, we have 
\begin{eqnarray*}
f_1+h\ast f_2
\!\!\!\!&=&\!\!\!\! 
H_1^{-1} \ast u \ast (H_1 \ast u^{-1} \ast f_1  + g_2 \ast f_2) \\
\!\!\!\!&=&\!\!\!\! 
H_1^{-1} \ast u \ast (H_1 \ast F_1 +g_2 \ast f_2)
\\
\!\!\!\!&=&\!\!\!\! 
 H_1^{-1} \ast u \ast (\varepsilon +(H_1-G_1)\ast F_1) ,
 \end{eqnarray*}
 which is invertible in $\calA(\bbp)$, as it is the product of invertible elements from $\calA(\bbp)$. 
\end{proof}

\vspace{-0.51cm}

 \section{$\calA(\bbp)$ is a Hermite ring, but not a projective free ring}
 
 \noindent The study of Serre's conjecture and algebraic $K$-theory naturally led to the notion of Hermite rings; see e.g., \cite{Lam}.
 
 \begin{definition}
 Let $R$ be a commutative unital ring. The ring $R$ is called 
{\em Hermite} if every finitely generated stably free $R$-module is free. The ring $R$ is  called {\em projective-free} if every 
 finitely generated projective $R$-module is free. 
\end{definition}

\vspace{-0.3cm}

 \noindent If $R$-modules $M,N$ are isomorphic, then we write $M\cong N$.  
 Recall that if $M$ is a finitely generated
  $R$-module, then:

\noindent $\bullet$  $M$ is called {\em free} if $M \cong R^k$ for some integer $k\geq 0$.

\noindent $\bullet$ $M$ is called {\em projective} if there exists an $R$-module $N$
  and an integer $\phantom{\bullet}$ $ m\geq 0$ such that $M\oplus  N \cong R^m$.
  
\noindent $\bullet$ $M$  is called {\em stably free}  if there exist  free  finitely generated $R$-modules $\phantom{\bullet }$ $F$ and $G$  such that $M \oplus F \cong G$.

\noindent It is clear that every projective free ring is Hermite. 

 For $m,n\in \mN$, $R^{m\times n}$ denotes the set of matrices with $m$ rows and $n$ columns having entries from $R$. The identity element in $R^{k\times k}$ having diagonal elements $1$,  and zeroes elsewhere will be denoted by $I_k$.  

In terms of matrices, $R$ is Hermite if and only if left-invertible tall matrices over $R$ can be completed to invertible ones (see e.g.  \cite[p.VIII]{Lam}, \cite[p.1029]{Tol}): 
For all $k, K\in \mN$ such that $k<K$, and for all $f\in R^{K\times k}$ 
such that there exists a $g\in R^{k\times K}$ so that $gf=I_k$, 
there exists an $f_c\in R^{K\times (K-k)}$ and there  exists a $G\in R^{K\times K}$ 
such that $G[\begin{smallmatrix}f & f_c\end{smallmatrix}]=I_K$.  
 
 In terms of matrices (see e.g. \cite[Proposition~2.6]{Coh} or
\cite[Lemma~2.2]{BRS}), the ring $R$ is projective-free if and
only if every idempotent matrix $P$ is conjugate (by an invertible
matrix $S$) to a diagonal matrix with elements $1$ and $0$ on the diagonal,
that is, for every $m\in \mN$ and every $P\in R^{m\times m}$
satisfying $P^2=P$, there exists an $S\in R^{m\times m}$ such that $S$
is invertible as an element of $R^{m\times m}$, and for some integer $r\geq 0$, 
$
S^{-1} P S=[\begin{smallmatrix}
I_r & 0\\
0 & 0
\end{smallmatrix}].
$
In 1976, it was shown independently by Quillen and Suslin,  that if $\mF$ is a field, then the polynomial
ring $\mF[x_1, \cdots , x_n]$  is
projective-free, settling Serre's conjecture from 1955 (see \cite{Lam}). 

\subsection{$\calA(\bbp)$ is a Hermite ring}

\noindent 
It is known that a commutative unital ring having Bass stable rank $\le 2$ is Hermite (see e.g., \cite[Corollary~36.17]{MorRup}).  In light of this result and Theorem~\ref{theorem_bsr_Ap}, we have:
 
 \begin{corollary}
 $\calA(\bbp)$ is a Hermite ring.
 \end{corollary}
 
 \subsection{$\calA(\bbp)$ is not a projective-free ring}
 
\noindent 
While every projective free ring is Hermite, the converse may not hold. In fact $\calA(\bbp)$ is such an example, by using the  matricial characterisation of projective free rings.

\begin{theorem}
$\calA(\bbp)$ is not projective free.
\end{theorem}
\begin{proof} Suppose $\calA(\bbp)$ is projective free. Let 
 $$
 \begin{array}{c}
P=\sum\limits_{m=0}^\infty  \frac{1}{\bbp(2m)} z^{2m}.
\end{array}
$$
Then $P\in \calA(\bbp)$, and 
$$
\begin{array}{rcl}
\widehat{P\ast P}(n)
\!\!\!&=&\!\!\!
\bbp(n) \widehat{P}(n) \widehat{P}(n)
\\[0.1cm]
\!\!\!&=&\!\!\!
\left\{ \begin{array}{l} \bbp(2m)  \frac{1}{(\bbp(2m))^2}=\frac{1}{\bbp(2m)}
\textrm{ if } n=2m,\;m\in \mN_0\\[0.1cm]
 \bbp(2m) \;\! 0^2 =0 \textrm{ if } n=2m+1, \;m\in \mN_0. 
 \end{array}\right.
 \end{array}
 $$
 So $P\ast P=P$. As we have assumed $\calA(\bbp)$ is projective free,  there is an  $f\in \{0,\epsilon\}$, $D=[\begin{smallarray}{c}f\end{smallarray}]$, and $S,S^{-1}\in \calA(\bbp)$ 
 such that $P=S^{-1} \ast D\ast S$. But then $P=0$ or $P=\varepsilon$, and either case is false. Consequently, $\calA(\bbp)$ is not projective free.  
\end{proof}

\noindent For a Banach space $A$, we denote the set of all continuous linear maps from $A$ to $\mC$ by $A^*$. Recall that for a commutative unital complex Banach algebra $A$,  the maximal ideal space $M(A)\subset A^\ast$  
 is the set of nonzero homomorphisms $A \!\rightarrow\! \mC\;$, endowed with the Gelfand topology, the weak-$\ast$ topology of  $A^\ast$. It is a compact Hausdorff space contained in the unit sphere of $A^\ast$.    Contractibility of the maximal ideal space $M(A)$ in the Gelfand topology suffices for $A$ to be projective-free (see, e.g., \cite[Corollary~1.4]{BruSas}). Thus   
the maximal ideal space $M(\calA(\bbp))$ is not contractible.

\vspace{-0.45cm}

\section{Idempotents in $\calA(\bbp)$}

\noindent The following result characterises idempotents in $\calA(\bbp)$.

\begin{theorem}
\label{11_3_2023_1416}
$f\in \calA(\bbp)$ is an idempotent if and only if 
for all $n\in \mN_0,$ $\widehat{f}(n)\in \{0, \frac{1}{\bbp(n)}\}$. 
\end{theorem}
\begin{proof} The `if' part is immediate from the definition of multiplication in $\calA(\bbp)$. If $f\in \calA(\bbp)$ is an idempotent, then $f\ast f=f$, and so 
$$
\begin{array}{c}
\sum\limits_{n=0}^\infty (\bbp(n)(\widehat{f}(n))^2-\widehat{f}(n)) z^n=(f\ast f)(z)-f(z)=0\quad (z\in \mC).
\end{array}
$$ 
Thus $\bbp(n)(\widehat{f}(n))^2-\widehat{f}(n)=0$ for all $n\in \mN_0$. So $\widehat{f}(n)\in \{0,\frac{1}{\bbp(n)}\}$. 
\end{proof}

\noindent Let $A$ be a commutative unital complex Banach algebra. The {\em Gelfand transform} of $a\in A$, defined by $\hat{a}(\pmb{\varphi}) := \pmb{\varphi}(a)$ for
 $\pmb{\varphi} \in  M(A)$, is a nonincreasing-norm morphism from $A$ into $C(M(A))$, the Banach algebra of complex-valued continuous functions on $M(A)$ equipped with the supremum norm $\|\cdot\|_\infty$ (given by $\|f\|_\infty:=\sup_{\varphi \in M(A)} |f(\varphi)|$ for all $f\in C(M(A))$.
 
From the special case of  (R3) on page~\pageref{5_3_2023_1335} when  $f=\varepsilon$ (the identity element of $\calA(\bbp)$), we have the following corona theorem:

\begin{proposition}
Let  $f_1,\cdots, f_n\in \calA(\bbp)$ $(n\in \mN)$. Then the following are equivalent:

\noindent {\em (1)} There exists a $\delta>0$ such that 
for all $k\in \mN_0,$ $\sum\limits_{i=1}^n |\widehat{f_i}(k)|\geq \frac{\delta}{\bbp(k)}$. 

\noindent {\em (2)}  There exist $g_1,\cdots, g_n\in \calA(\bbp)$ such that $\sum\limits_{i=1}^n 
g_i \ast f_i =\varepsilon$. 
\end{proposition}

\noindent For $k\in \mN_0$, let $\varphi_k$ denote the complex homomorphism given by   
$$
\begin{array}{c}
\varphi_k(f)=\widehat{f}(k)\textrm{ for all }f\in \calA(\bbp).
\end{array}
$$ 
From (R5) (p.~\pageref{11_3_2023_1509}), we know that $\varphi_k\in M(\calA(\bbp))$.  
From elementary Banach algebra theory (see e.g., \cite[Lemma~9.2.6]{Nik}), it follows that the 
 set $\{\varphi_k:k\in \mN_0\}$ is dense in $M(\calA(\bbp))$.

We will now show that $\{\varphi_k:k\in \mN_0\}$, with the topology induced from $M(\calA(\bbp))$, is homeomorphic to $\mN_0$, with the topology induced from $\mR$. In the proof below, we will use the notation $\delta_{m,n}:=0$ if $m\neq n$ and $\delta_{m,m}:=1$ for all $m,n\in \mN_0$. 

\begin{proposition}
For $k_0\in \mN_0,$ a net $(\varphi_{k_i})_{i\in I}$ converges to $\varphi_{k_0}$ in $M(\calA(\bbp))$ if and only if $(k_i)_{i\in I}$ is eventually constant$,$ equal to $k_0$. 
\end{proposition}
\begin{proof}`If' part: Let $(k_i)_{i\in I}$ be eventually constant, equal to $k_0$.  There exists an $i_*\in I$ such that for all $i\ge i_*$ (where $\ge$ denotes the order on the directed set $I$), $k_i=k_0$, and so  for all $f\in \calA(\bbp)$, we have 
that  $\varphi_{k_i}(f)\!=\!\widehat{f}(k_i)\!=\!\widehat{f}(k_0)\!=\! \varphi_{k_0}(f)$. So $(\varphi_{k_i})_{i\in I}$ converges to $\varphi_{k_0}$ in $M(\calA(\bbp))$.

\noindent `Only if' part: Suppose that $(\varphi_{k_i})_{i\in I}$ converges to $\varphi_{k_0}$ in $M(\calA(\bbp))$. Then with $f=z^{k_0}\in \calA(\bbp)$, by the definition of the Gelfand topology,  the net $(\varphi_{k_i}(z^{k_0}))_{i\in I}$ converges to $\varphi_{k_0}(z^{k_0})=1$ in $\mR$. Thus for $\epsilon=\frac{1}{2}>0$, there exists an $i_*\in I$ such that  $|\delta_{k_i , k_0}-1|=|\varphi_{k_i}(z^{k_0})-1|<\frac{1}{2}$ for all $i\ge i_*$, i.e., $k_{i}=k_0$.
\end{proof}

\noindent As $M(\calA(\bbp))$ is compact, while $\mN_0$ with its usual topology is not compact, it follows that not all elements of $M(\calA(\bbp))$ are of the form $\varphi_k$ for some $k\in \mN_0$. Explicit examples are known (see below), and these were mentioned as non-fixed maximal ideals in \cite[Remark, pp. 6-7]{vRen}.

\begin{example} 
Let $\bbk=(k_n)_{n\in \mN}$ be any subsequence of the sequence of natural numbers. Define 
$$
\begin{array}{c}
I_{\bbk}:=\{f\in \calA(\bbp):\lim\limits_{n\rightarrow \infty} \bbp_{k_n}\widehat{f}(k_n)=0\}.
\end{array}
$$
 Then $I_{\bbk}$ is an ideal of $\calA(\bbp)$. (It is clear that if $f,g\in I_{\bbk}$, then $f+g\in I_{\bbk}$. If $f\in I_{\bbk}$ and $g\in \calA(\bbp)$, then there exists a $C>0$ such that 
$$
\begin{array}{c}
|\widehat{g}(k)|\leq \frac{C}{\bbp(k)}\textrm{ for all }k\in \mN_0,
\end{array}
$$
 and so 
$$
\begin{array}{rcl}
0\leq \bbp(k_n)|\reallywidehat{(f\ast g)}(k_{n})|
\!\!\!&=&\!\!\!
\bbp(k_n)|\bbp(k_n)\widehat{f}(k_n)\widehat{g}(k_n)|\\[0.1cm]
\!\!\!&\leq &\!\!\!
\bbp(k_n)\bbp(k_n)|\widehat{f}(k_n)|\frac{C}{\bbp(k_n)}
\stackrel{n\rightarrow \infty}{\longrightarrow} 0.
\end{array}
$$ 
Thus $f\ast g\in \calA(\bbp)$.) Moreover, $I_{\bbk}\neq \calA(\bbp)$ since $\varepsilon\not\in I_{\bbk}$:
$$
\begin{array}{c}
\lim\limits_{n\rightarrow \infty} \bbp(k_n) |\widehat{\varepsilon}(k_n)|
=\lim\limits_{n\rightarrow \infty} \bbp(k_n) \frac{1}{\bbp(k_n)}=1\neq 0.
\end{array}
$$
Hence there exists a maximal ideal $M$ in $\calA(\bbp)$ such that $I_{\bbk}\subset M$. We note that for each $m\in \mN_0$, $M\neq \ker \varphi_{m}$ (since for any  $m\in \mN_0$, $f:=z^m\in I_{\bbk}\subset M$, and then $\varphi_m(f)=1\neq 0$). 
\hfill$\Diamond$
\end{example}

 \noindent Recall the Shilov idempotent theorem 
 (see, e.g., \cite[Corollary~6.5]{Gam}): If $E$ is an open-closed subset of $M(A)$, then there is a unique element $f$ of $A$ such that $f^2=f$ and $\hat{f}=\mathbf{1}_E$ (the characteristic function of $E$, which is identically equal to $1$ on $E$, and $0$ elsewhere on $M(A)\setminus E$).  From Theorem~\ref{11_3_2023_1416}, 
 we get the following. 
 
 \begin{corollary} 
  If $\bbp(n)\neq 1$ for all $n\in \mN_0,$ then $M(\calA(\bbp))$ is connected.
  \end{corollary}
  \begin{proof}
  If $E\subsetneq M(\calA(\bbp))$ is closed and open, then there exists an idempotent $f$ in $\calA(\bbp)$ such that $\hat{f}=\mathbf{1}_E$. If $n\in \mN_0$ is such that  $\varphi_n\in E$, we get $\{0,\frac{1}{\bbp(n)}\}\owns \widehat{f}(n)=\varphi_n(f)=\hat{f}(\varphi_n)=\bm{1}_E(\varphi_n)=1$, so that $\bbp(n)=1$, a contradiction. So $E$ does not contain any $\varphi_n$, $n\in \mN_0$, so that $M(\calA(\bbp))\setminus E$ contains $\{\varphi_n:n\in \mN_0\}$. As  $\{\varphi_n:n\in \mN_0\}$ is dense in $M(\calA(\bbp))$, we conclude that  $M(\calA(\bbp))\setminus E=M(\calA(\bbp))$, i.e., $E=\emptyset $. Thus  $M(\calA(\bbp))$ is connected.
  \end{proof}

\section{Exponentials in $\calA(\bbp)$}

\noindent We will show that every invertible element in $\calA(\bbp)$ has a logarithm. 

\begin{lemma}
If $f\in \calA(\bbp),$ then for all $z\in \mC,$ $|f(z)|\leq \varepsilon(|z|) \|f\|$.
\end{lemma}
\begin{proof} As $|\widehat{f}(k)|=\frac{\bbp(k) |\widehat{f}(k)|}{\bbp(k)} \leq \frac{\|f\|}{\bbp(k)}$ for all $k\in \mN_0$, we have 
$$
\begin{array}{c}
|f(z)| =|\sum\limits_{k=0}^\infty \widehat{f}(k) z^k| 
\leq \sum\limits_{k=0}^\infty \frac{\|f\|}{\bbp(k)}|z|^k =\|f\|\varepsilon(|z|).
\end{array} \eqno\qedhere
$$
\end{proof}

\noindent It follows that if $(f_n)_{n\in \mN}$ is a convergent sequence in $\calA(\bbp)$, then it converges pointwise. 

\begin{lemma}
\label{7_3_2023_1748}
If $f\in \calA(\bbp),$ then $(e^f)(z)=\sum\limits_{k=0}^\infty 
\frac{e^{\bbp(k) \widehat{f}(k)}}{\bbp(k)} z^k$ for all $z\in \mC$. 
\end{lemma}
\begin{proof} We note that $(f^n)(z)=\sum\limits_{k=0}^\infty \bbp(k)^{n-1} (\widehat{f}(k))^n z^k$ ($z\in \mC$). We have 
$$
\begin{array}{rcl}
\sum\limits_{n=0}^\infty\sum\limits_{k=0}^\infty 
|\frac{\bbp(k)^{n-1}(\widehat{f}(k))^n }{n!} z^k|
\!\!\!&\leq &\!\!\!
\sum\limits_{n=0}^\infty\sum\limits_{k=0}^\infty  \frac{|z|^k}{\bbp(k)} \frac{\|f\|^n}{n!}
= 
\sum\limits_{k=0}^\infty\frac{|z|^k}{\bbp(k)}\sum\limits_{n=0}^\infty\frac{\|f\|^n}{n!} 
\\[0.3cm]
\!\!\!&=&\!\!\!
\sum\limits_{k=0}^\infty\frac{|z|^k}{\bbp(k)} e^{\|f\|} \\[0.39cm]
\!\!\!&=&\!\!\! e^{\|f\|} \varepsilon (|z|)<\infty.
\end{array}
$$
So exchange of summations below is allowed. We have
$$
\phantom{AAAAAAAAAAAA}
\begin{array}{rcl}
(e^f)(z)\!\!\!&=&\!\!\!\sum\limits_{n=0}^\infty \frac{(f^n)(z)}{n!} 
=\sum\limits_{n=0}^\infty\sum\limits_{k=0}^\infty 
\frac{\bbp(k)^{n-1}(\widehat{f}(k))^n }{n!}  z^k
\\
\!\!\!&=&\!\!\!\sum\limits_{k=0}^\infty \frac{z^k}{\bbp(k)} 
\sum\limits_{n=0}^\infty \frac{(\bbp(k)\widehat{f}(k))^n}{n!}
\\
\!\!\!&=&\!\!\!\sum\limits_{k=0}^\infty \frac{z^k}{\bbp(k)}  e^{\bbp(k) \widehat{f}(k)}.
\phantom{AAAAAAAAAAA}
\qedhere
\end{array}
$$
\end{proof}

\noindent Let $\calA(\bbp)^{-1}$ denote the multiplicative group of invertible elements of $\calA(\bbp)$, and $e^{\calA(\bbp)}$ the subgroup of $\calA(\bbp)$ consisting of all exponentials $e^f$, $f\in \calA(\bbp)$. 

\begin{theorem}
$e^{\calA(\bbp)}=\calA(\bbp)^{-1}$.
\end{theorem}
\begin{proof} Let $g\in \calA(\bbp)^{-1}$. By (R1) on page~\pageref{3320231845}, there exists a $\delta>0$ such that $|\widehat{g}(k)|\geq \frac{\delta}{\bbp(k)}$ for all $k\in \mN_0$. In particular, $\widehat{g}(k)\neq 0$, and 
we define 
$$
\begin{array}{c}
\widehat{f}(k):=\frac{1}{\bbp(k)} \textrm{Log}(\bbp(k) \widehat{g}(k)), \quad k\in \mN_0.
\end{array}
$$
where $\textrm{Log}:\mC\setminus\{0\} \rightarrow \mR\times (-\pi,\pi]$ denotes the principal branch of the logarithm.  We have $0<\delta \leq \bbp(k)|\widehat{g}(k)|\leq \|g\|$ for all $k\in \mN_0$, and so 
$$
\begin{array}{c}
\sup_{k\in \mN_0} \bbp(k) |\widehat{f}(k)|\leq \sqrt{(\max\{\log \delta, \, \log \|g\|\})^2+\pi^2}=:C<\infty,
\end{array}
$$
showing that $f\in \calA(\bbp)$. It follows from Lemma~\ref{7_3_2023_1748} that $e^f=g$.
\end{proof}

\noindent For a topological space $X$, let  $H^1(X,\mZ)$ denote the first \v{C}ech cohomology group of $X$ with integer coefficients. For background on \v{C}ech cohomology, see \cite{EilSte}. For a   commutative unital complex 
semisimple Banach algebra $A$, the quotient group $A^{-1}/e^{A}$ is isomorphic to $H^1(M(A), \mZ)$ (see e.g. \cite[Corollary~7.4]{Gam}). Thus $H^1(M(\calA(\bbp)), \mZ)=\{0\}$.

\section{Solvability of $Ax=b$}

\noindent We will show the following:

\begin{theorem}
\label{4320231947}
Let $A \in \calA(\bbp)^{m\times n},$ $b \in  \calA(\bbp)^{m\times  1}$.  
\phantom{Then the following }
Then the following  are equivalent:
\begin{enumerate}
 \item There exists an $x  \in \calA(\bbp)^{n\times 1}$ such that $A\ast x = b$.
 \item There exists a $\delta>0$ such that for all $k\in \mN_0$ and all $y \in \mC^{m\times 1},$ $\nm (\widehat{A}(k))^* y\nm_2 \geq  \delta |\langle y, \widehat{b}(k)\rangle_2|$.
 \end{enumerate}
\end{theorem}

\vspace{-0.2cm}

\noindent Here $\langle \cdot,\cdot \rangle_2$ denotes the usual Euclidean inner product on $\mC^{\ell \times 1}$ for $\ell\in \mN$, and $\nm \cdot\nm_2$ is the corresponding induced norm. Also, if $a_{ij}\in \calA(\bbp)$ denotes the entry in the $i^{\textrm{th}}$ row and $j^{\textrm{th}}$ column of $A$, then $\widehat{A}(k)\in \mC^{m\times n}$ is the matrix whose entry in the $i^{\textrm{th}}$ row and $j^{\textrm{th}}$ column is $\widehat{a_{ij}}(k)$, $1\leq i \leq m$, $1\leq j\leq n$, $k\in \mN_0$. For a matrix $M\in \mC^{m\times n}$, $M^*$ denotes its Hermitian adjoint (obtained by taking the entrywise complex conjugate of $M$ and then taking the transpose of the resulting matrix). We will use the following elementary linear algebraic result; see \cite[Lemma~8.2]{Sas}. We include the short proof for the sake of completeness. 

\begin{lemma}
\label{4_3_2023_1958}
 Let $ A\in \mC^{m\times n}$ and $b\in \mC^{m\times 1}$ be such that 
$$
 \exists \delta>0 \textrm{ such that  }
 \textrm{ for all }y \in \mC^{m\times 1},\, 
  \nm A^\ast y\nm_2 \geq  \delta | \langle y,b \rangle_2 |.\quad \quad  (\star)
  $$
   Then there exists an $x \in \mC^{n\times 1}$ such that $Ax = b$ with $\nm x\nm_2 \leq  \frac{1}{\delta}$.
\end{lemma}

\vspace{-0.6cm}

\begin{proof}
 For $y\in \ker A^\ast$,  ($\star$) implies $\langle y, b\rangle_2 =0$. So  $b\in(\ker A^\ast)^\perp = \textrm{ran}\;\!A$. If $y \in \ker(AA^\ast)$, 
 then 
 $
 \nm A^\ast y\nm_2^2 \!=\! \langle A^\ast y, A^\ast y \rangle_2  
 \!=\! \langle AA^\ast y,y \rangle_2  \!=\! \langle 0,y \rangle_2  \!=\! 0.
 $  
 Thus $A^\ast y = 0$, and so $y \in \ker A^\ast = (\textrm{ran}\;\!A)^\perp$. 
 Since we had shown above that $b \in \textrm{ran}\;\!A$, we have 
 $\langle b,y \rangle_2  = 0$. As $y \in \ker(AA^\ast)$ 
 was arbitrary,  $b \in (\ker (AA^\ast) )^\perp = \textrm{ran}((AA^\ast )^\ast) = \textrm{ran}(AA^\ast )$. Hence there exists a
$y_0 \in\mC^{m\times 1}$ such that $AA^\ast y_0 =b$. 
Taking $x:=A^\ast y_0 \in \mC^{n\times 1}$, we have $Ax=b$. 
If $b = 0$, then we can take $x = 0$, and the estimate on $\nm x\nm_2$ 
is obvious. So we assume that $b\neq 0$ and then $A^\ast y_0\neq 0$ (since we know that $AA^\ast y_0 =b$). 
We have, using the given inequality ($\star$), 
$$
\nm A^\ast y_0\nm_2^2 \!=\! \langle A^\ast y_0,A^\ast y_0 \rangle_2  
\!=\! \langle y_0,AA^\ast y_0 \rangle_2  \!=\! \langle y_0,b \rangle_2  
\!=\! |\langle y_0,b \rangle_2 | \!\leq\!  \frac{\nm A^\ast y_0\nm_2}{\delta} .
$$ 
Since $A^\ast y_0\neq 0$, we obtain $\nm x \nm_2 = \nm A^\ast y_0\nm_2 \leq \frac{1}{\delta}$.
\end{proof}

\begin{proof}[Proof of Theorem~\ref{4320231947}] (1)$\Rightarrow$(2): 
As $x\in \calA(\bbp)^{n\times 1}$, there exists a $C>0$ such that for all $k\in \mN_0$, $\nm \widehat{x}(k)\nm_2 \leq \frac{C}{\bbp(k)}$. So for all $y\in \mC^{m\times 1}$ and $k\in \mN_0$, 
$$
\begin{array}{rcl}
|\langle y ,\widehat{b}(k)\rangle_2 | 
\!\!\!&=&\!\!\!
|\langle y , \bbp(k) \widehat{A}(k) \widehat{x}(k)\rangle_2|
=\bbp(k) |\langle (\widehat{A}(k))^*y, \widehat{x}(k)\rangle_2|
\\[0.1cm]
\!\!\!&\leq &\!\!\! \bbp(k) \nm (\widehat{A}(k))^*y\nm_2 \frac{C}{\bbp(k)} \quad \textrm{(Cauchy-Schwarz)}\\[0.1cm]
\!\!\!&=&\!\!\! 
C\nm (\widehat{A}(k))^*y\nm_2 .
\end{array}
$$
Setting $\delta := \frac{1}{C}> 0$ and rearranging gives (2).

\noindent (2)$\Rightarrow$(1): Fix $k\in \mN_0$. The condition in statement (2) and Lemma~\ref{4_3_2023_1958}
   implies the existence of an $ x_k \in \mC^{n\times 1}$ 
   such that $\widehat{A}(k)x_k = \widehat{b}(k)$, with 
   $\nm x_k \nm_2 \leq \frac{1}{\delta}$. In this way, we obtain a sequence $(x_k)_{k\geq 0}$ in $\mC^{n\times 1}$. Let the components  of $x_k$ be denoted by    $x_k^{(1)},\cdots, x_k^{(n)}\in \mC$. Define 
   $$
   \begin{array}{c}
   x^{(i)}(z)=\sum\limits_{n=0}^\infty \frac{x^{(i)}_k}{\bbp(k)}z^k\,\,
   \textrm{ for all }z\in \mC,\, 1\leq i\leq n.
   \end{array}
   $$ 
   Then each $x^{(i)}\in \calA(\bbp)$ (as $|\frac{x^{(i)}}{\bbp(k)}|\leq \frac{1/\delta}{\bbp(k)}$ for all $k\in \mN_0$), and so the column vector $x$ having these components belongs to $\calA(\bbp)^{n\times 1}$. Also,   
   $$
   \begin{array}{c}
\reallywidehat{(A\ast x)}(k)=\bbp(k) \widehat{A}(k) \frac{1}{\bbp(k)} x_k
   =\widehat{b}(k)\textrm{ for all }k\in \mN_0,
   \end{array}
   $$
    i.e., $A\ast x=b$. 
\end{proof}

\vspace{-0.3cm}

\section{Krull dimension of $\calA(\bbp)$ is infinite}

\vspace{-0.1cm}

\begin{definition}
The {\em Krull dimension} of a commutative ring $R$ is the 
supremum of the lengths of chains of distinct proper prime ideals of
$R$.
\end{definition}

\vspace{-0.1cm}

\noindent Recall that the Hardy algebra $H^\infty$ is the Banach algebra of bounded and holomorphic functions on the unit disc $\mD:=\{z\in \mC: |z|<1\}$, with pointwise operations and the supremum norm $\|\cdot\|_\infty$. In  \cite{vonRen77}, von~Renteln showed that the Krull dimension of 
$H^\infty$ is infinite. We adapt the idea given in \cite{vonRen77}, 
to show that the Krull dimension of $\calA(\bbp)$ is infinite too. 
A key ingredient of the proof in \cite{vonRen77} was the use of a canonical factorisation of $H^\infty$ elements used to create ideals with zeroes at prescribed locations with prescribed orders. Instead of zeroes of our entire functions, we will look at the indices for the vanishing coefficients in the Taylor expansion centred at $0$, and instead of orders of zeroes, we will use the notion of `index-order' introduced below. 

If $f\in \calA(\bbp)$ and $k\in \mN_0$  is an index such that $\widehat{f}(k)=0$, then we define the {\em index-order $m(f,k)$ of the index $k$ for $f$}  by 
$$
m(f,k)=\max\{m \in \mN_0: \widehat{f}(k+\ell)=0 \textrm{ whenever } 0\leq \ell \leq m-1\}.
$$
If $\widehat{f}(k+\ell)=0$ for all $\ell \in \mN_0$, then we set $m(f,k)=\infty$. If $\widehat{f}(k)\neq 0$, then we set $m(f,k)=0$. 
Analogous to the order of a zero of a holomorphic function, the index-order satisfies the following property. 
$$ 
\textrm{(P1): If }f,g\in \calA(\bbp),\,k\in \mN_0\textrm{, then }m(f+g,k)\geq \min\{m(f,k), m(g,k)\}.
$$
The order of a zero $\zeta$ of the pointwise product of two holomorphic functions is the sum of the orders of $\zeta$ as a zero of each of of the two holomorphic functions. For the index order, and the weighted Hadamard product, we have the following instead:
$$ 
\textrm{(P2): If }f,g\in \calA(\bbp),\,k\in \mN_0\textrm{, then }m(f\ast g,k)\geq \max\{m(f,k), m(g,k)\}.
$$
We will use the following known result; see \cite[Theorem, \S 0.16, p.6]{GilJer60}.

\begin{proposition}
\label{thm_GilmanJerison}
If $J$ is an ideal in a ring $R,$ and $M$ is a set that is closed
under multiplication and $M\cap J=\emptyset$, then there exists an
ideal $P$ such that $J\subset P$ and $P\cap M=\emptyset,$ and $P$
maximal with respect to these properties. Moreover$,$ such an ideal $P$
is necessarily prime.
\end{proposition}

\begin{theorem}
\label{theorem_Krull}
The Krull dimension of $\calA(\bbp)$ is infinite.
\end{theorem}
\begin{proof} 
Let $a_{k}=2^k$ for all $k\in \mN_0$. Let $n\in \mN$. Define $f_n\in \calA(\bbp)$ by 
$$
\left\{
\begin{array}{rcll}
\widehat{f_n}(a_k+\ell)\!\!\!&=&\!\!\!0 &\textrm{ whenever }0\leq \ell \leq k^{n+1},\\[0.1cm]
\widehat{f_n}(m) \!\!\!&=&\!\!\!\frac{1}{\bbp(m)}& \textrm{ if } m\not\in\bigcup_{k\in \mN_0} \{ a_k+\ell:  0\leq \ell \leq k^{n+1}\}.
\end{array}
\right. 
$$
Note that $m(f_n,a_k)\geq k^{n+1}$, but 
for each fixed $n\in \mN$, there exists a $K_n\in \mN_0$ such that  the gap between the indices, 
$$
a_{k+1}-a_k=2^{k+1}-2^k =2^k >k^{n+1}\textrm{ for all }k>K_n,
$$
 and so $m(f_n,a_k)=k^{n+1}$ for all  $k>K_n$. 
Hence 
\begin{equation}
\label{function_family}
\begin{array}{c}
\lim\limits_{k\rightarrow \infty} \frac{m(f_{n},a_{k})}{k^{n}}=\infty
\quad \textrm{ and } \quad
\lim\limits_{k\rightarrow \infty} \frac{m(f_{n},a_{k})}{k^{n+1}}=1<\infty.
\end{array}
\end{equation}
Let  $
I:=\{f\in \calA(\bbp) : \exists k_{0}(f) \in \mN_0 \textrm{ such that }
\forall k>k_{0}(f), \; \widehat{f}(a_{k})=0\}.
$ 
The set $I$ is nonempty since $0\in I$. Clearly $I$ is closed under
addition, and $f\ast g\in I$ whenever $f\in I$ and $g\in \calA(\bbp)$. So $I$ is
an ideal of $\calA(\bbp)$. For $n\in \mN$, we define
$$
\begin{array}{rcl}
I_{n}\!\!\!&=&\!\!\!
\{ f\in I : \lim\limits_{k\rightarrow \infty}
\frac{m(f,a_{k})}{k^{n}}=\infty\},\\
M_{n}\!\!\!&=&\!\!\!
\{ f\in \calA(\bbp): \sup\limits_{k\in \mN}
\frac{m(f,a_{k})}{k^{n}}<\infty\}. 
\end{array}
$$
Clearly $f_{n} \in I_{n}$, and so $I_{n}$ is not empty. Using
(P1), we see that if $f,g\in I_{n}$, then $f+g\in I_{n}$.
If $g\in\calA(\bbp)$ and $f\in I_n$, then (P2) impies that $f\ast g\in I_{n}$. Hence $I_{n}$ is an ideal of $\calA(\bbp)$.

The identity element $\varepsilon \in M_{n}$  for all $n\in \mN$. If $f,g \in 
M_{n}$, then it follows from (P2) that $f\ast g\in M_{n}$. Thus
$M_{n}$ is a nonempty multiplicatively closed subset of $\calA(\bbp)$.

It is easy to check that for all $ n\in \mN$, $I_{n+1}\subset I_{n}$  and $M_{n}\subset M_{n+1}$.  
We now prove that the inclusions are strict for each $n\in \mN$. From
\eqref{function_family}, it follows that $f_{n} \in I_{n}$ but $f_{n}
\not\in I_{n+1}$. Also $f_{n} \in M_{n+1}$ and $f_{n} \not\in M_{n}$.

Next we show that $I_{n} \cap M_{n} =\emptyset$. Indeed, if $f\in
I_{n} \cap M_{n}$, then
$$
\begin{array}{c}
\infty
= \lim\limits_{k\rightarrow \infty} \frac{m(f,a_{k})}{k^{n}}
= \limsup\limits_{k\rightarrow \infty} \frac{m(f,a_{k})}{k^{n}}
\leq \sup\limits_{k\in \mN} \frac{m(f,a_{k})}{k^{n}}
< \infty,
\end{array}
$$
a contradiction. But $I_{n} \cap M_{n+1} \neq
\emptyset$, since $f_{n} \in I_{n} $ and $f_n\in M_{n+1}$.

We will now show that the Krull dimension of $\calA(\bbp)$ is infinite by showing that 
for all $N\in \mN$, we can construct a chain of strictly decreasing prime ideals 
 $
P_{N+1}\subsetneq P_{N}\subset \cdots \subsetneq P_2 \subsetneq P_1
$ in $\calA(\bbp)$.

Fix an $N\in \mN$. Applying Proposition~\ref{thm_GilmanJerison}, taking $J=I_{N+1}$ and
$M=M_{N+1}$, we obtain the existence of a prime ideal $P=P_{N+1}$ in
$\calA(\bbp)$, which satisfies $I_{N+1}
\subset P_{N+1}$ and $P_{N+1} \cap M_{N+1} =\emptyset$.

We claim the ideal $I_{N}+P_{N+1}$ of $\calA(\bbp)$ satisfies $(I_{N} +P_{N+1})
\cap M_{N}=\emptyset$. Let $h=f+g \in I_{N}+P_{N+1}$, where $f\in
I_{N}$ and $g\in P_{N+1}$. Since $g\in P_{N+1}$, by the construction
of $P_{N+1}$ it follows that $g\not\in M_{N+1}$. But $M_{N} \subset
M_{N+1}$, and so $g\not\in M_{N}$ as well. Thus there exists a
subsequence $(k_{\ell})_{\ell\in \mN_0}$ of $(k)_{k\in \mN_0}$ such that
$
\lim_{\ell\rightarrow \infty} \frac{m(g, a_{k_{\ell}})}{k_{\ell}^{N}}=\infty.$
From (P1), we obtain $\frac{m(h, a_{k_{\ell}})}{k_{\ell}^{N}}
\geq
\min \{
\frac{m(f, a_{k_{\ell}})}{k_{\ell}^{N}},
 \frac{m(g, a_{k_{\ell}})}{k_{\ell}^{N}}\}$. As $f\in I_N$,  it follows that 
$$
\begin{array}{c}
\sup\limits_{\ell\in \mN} \frac{m(h, a_{k_{\ell}})}{k_{\ell}^{N}}
\geq 
\min \Big\{
\limsup\limits_{\ell \rightarrow \infty} \frac{m(f, a_{k_{\ell}})}{k_{\ell}^{N}},
\limsup\limits_{\ell\rightarrow \infty} \frac{m(g, a_{k_{\ell}})}{k_{\ell}^{N}}
\Big\}
\geq \infty.
\end{array}
$$
Thus $h\not\in M_{N}$. Consequently, $(I_{N} +P_{N+1}) \cap M_{N}=\emptyset$.

Clearly $I_{N} \subset I_{N}+P_{N+1}$. 
Applying Proposition~\ref{thm_GilmanJerison} again, now taking $J=I_{N}+P_{N+1}$ and
$M=M_{N}$, we obtain the existence of a prime ideal $P=P_{N}$ in
$\calA(\bbp)$ such that $I_{N}+P_{N+1} \subset P_{N}$ and $P_{N} \cap M_{N} =\emptyset$. 
Thus $P_{N+1}\subset I_N+ P_{N+1}\subset P_N$.
The first inclusion is strict because $f_N\in I_N\subset I_N+P_{N+1}$. But $f_N\not\in P_{N+1}$ (since $f_N\in M_{N+1}$ and 
$P_{N+1}\cap M_{N+1}
=\emptyset$ by the construction of $P_{N+1}$). Thus $P_{N+1}\subsetneq P_N$. 

Now consider the ideal $J:=I_{N-1}+P_N\supset I_{N-1}$ of $\calA(\bbp)$ and the multiplicatively closed set $M:=M_{N-1}$ of $\calA(\bbp)$. Similar to the argument given above, then\footnote{Let $h=f+g \in I_{N-1}+P_{N}$, where $f\in
I_{N-1}$ and $g\in P_{N}$. Since $g\in P_{N}$, by the construction
of $P_{N}$, $g\not\in M_{N}\supset M_{N-1}$, and so $g\not\in M_{N-1}$. Thus there exists a subsequence $(k_{\ell})_{\ell\in \mN_0}$ of $(k)_{k\in \mN_0}$ such that
$\lim\limits_{\ell\rightarrow \infty} \frac{m(g, a_{k_{\ell}})}{k_{\ell}^{N-1}}=\infty$. 
As $f\in I_{N-1}$,  
$
\sup\limits_{\ell\in \mN} \frac{m(h, a_{k_{\ell}})}{k_{\ell}^{N-1}}
\geq 
\min \{
\limsup\limits_{\ell \rightarrow \infty} \frac{m(f, a_{k_{\ell}})}{k_{\ell}^{N-1}},
\limsup\limits_{\ell\rightarrow \infty} \frac{m(g, a_{k_{\ell}})}{k_{\ell}^{N-1}}
\}
\geq \infty.
$
Thus $h\not\in M_{N-1}$. So $(I_{N-1} +P_{N}) \cap M_{N-1}=\emptyset$.} 
$J\cap M=(I_{N-1}+P_N)\cap M_{N-1}=\emptyset$. 
By Proposition~\ref{thm_GilmanJerison}, taking $J=I_{N-1}+P_N\supset I_{N-1}$  and $M=M_{N-1}$, there exists a prime ideal $P=P_{N-1}$ in $\calA(\bbp)$ such  that 
$I_{N-1}+P_{N} \subset P_{N-1}$ and $P_{N-1} \cap M_{N-1} =\emptyset$. 
Thus $P_{N}\subset I_{N-1}+ P_{N}\subset P_{N-1}$, and again the first inclusion is strict (because $f_{N-1}\in I_{N-1}\subset I_{N-1}+P_N$, $f_{N-1}\in M_N$ and $M_N\cap P_N=\emptyset$). 

Proceeding in this manner, we obtain the chain of distinct prime ideals
 $
P_{N+1} \subsetneq P_{N} \subsetneq P_{N-1} \subsetneq \cdots \subsetneq P_1.
$ in $\calA(\bbp)$. 
As  $N\in \mN$ was arbitrary, it follows that the Krull dimension of $\calA(\bbp)$ is infinite.
\end{proof}

\section{$\calA(\bbp)$ is neither Artinian nor Noetherian}

\vspace{-0.2cm}

\noindent Even Noetherian rings can have an infinite Krull dimension (see e.g. \cite[Appendix, Example~E1]{Nag} or \cite[Exercise~9.6]{Eis}). 
However, in our case,  $\calA(\bbp)$ is not Noetherian.

\begin{definition}
A commutative ring $R$ is called {\em Noetherian} if there is no infinite increasing sequence of ideals, that is, for every increasing sequence $I_1\subset I_2 \subset I_3 \subset \cdots$ of ideals of  $R$, there exists an $N\in \mN$ such that $I_n=I_N$ for all $n>N$.
A commutative ring $R$ is called {\em Artinian} if there is no infinite descending sequence of ideals, that is, for every decreasing sequence $I_1\supset I_2 \supset I_3 \supset \cdots$ of ideals of $R$, there exists an $N\in \mN$ such that $I_n=I_N$ for all $n>N$.  
\end{definition}

\vspace{-0.2cm}

\noindent $\calA(\bbp)$ is not Noetherian. (Let 
 $I_n=\{f\in \calA(\bbp): \widehat{f}(k)=0\textrm{ for all } k\geq n\}$ for all $n\in \mN$. Clearly $I_1\subset I_2\subset I_3\subset \cdots$. Moreover, each inclusion is strict, since if $f_n:=z^{n}\in \calA(\bbp)$, $n\in \mN$, then $f_n\in I_{n+1}\setminus I_{n}$.) 

\noindent $\calA(\bbp)$ is not Artinian. (Let 
 $
I_n=\{f\in \calA(\bbp): \widehat{f}(k)=0\textrm{ for all } k\leq n\}$ for all $n\in \mN$.  
Clearly $I_1\supset I_2\supset I_3\supset \cdots$. Moreover, each inclusion is strict, since if $f_n:=z^{n+1}\in \calA(\bbp)$, $n\in \mN$,  then  $f_n\in I_n\setminus I_{n+1}$.)

\vspace{-0.3cm}

\section{$\calA(\bbp)$ is a coherent ring}

\vspace{-0.1cm}

\noindent  In absence of the Noetherian `finiteness' property, a natural finiteness  question is that of coherence. We refer the reader to the article \cite{Gla} and the monograph \cite{Gla0} for the relevance of the property of
coherence in commutative holomogical algebra.

\begin{definition}
A commutative unital ring $R$ is called {\em coherent} if every 
finitely generated ideal $I$ is finitely presentable, that is, 
there exists an exact sequence
$0 \rightarrow K \rightarrow F  \rightarrow I \rightarrow 0$ of $R$-modules, 
where $F$ is a finitely generated free $R$-module and $K$ is a finitely generated
$R$-module.
\end{definition}

\vspace{-0.1cm}

\noindent All Noetherian rings are coherent, but not all coherent rings are Noetherian. (For example, the polynomial ring $\mC[x_1, x_2, x_3, \cdots ]$ is not Noetherian (because the sequence of ideals $\langle x_1 \rangle \subset  \langle x_1, x_2\rangle  \subset   \cdots$ is ascending and not stationary), but $\mC[x_1, x_2, x_3, \cdots ]$ is coherent \cite[Corollary~2.3.4]{Gla0}.)

A commutative ring in which every finitely generated ideal is principal is said to be {\em B\'ezout}. By property (R4) (p.~\pageref{3320231651}), $\calA(\bbp)$ is a B\'ezout ring. It is known that B\'ezout {\em domains} are coherent, but we cannot use this to conclude that $\calA(\bbp)$ is coherent, since $\calA(\bbp)$ is {\em not} a domain (as $\calA(\bbp)$ has nontrivial zero divisors: e.g. $z\ast z^2=0$). 

\begin{theorem}
$\calA(\bbp)$ is coherent.
\end{theorem}

\vspace{-0.5cm}

\begin{proof} Let $I$ be a finitely generated ideal in $\calA(\bbp)$. 
Then $I$ is principal by the property (R4) (p.~\pageref{3320231651}). 
 So there exists an $f_I\in \calA(\bbp)$ such that $I =\langle f_I \rangle$. 
 Define $\chi \in \calA(\bbp)$ by setting 
 $$
 \widehat{\chi}(k)=\left\{ \begin{array}{cl} 
 \frac{1}{\bbp(k)} & \textrm{if } \widehat{f_I}(k)=0,\\
 0 & \textrm{if } \widehat{f_I}(k)\neq 0.
 \end{array}
 \right.
 $$
 Define $K=\langle \chi \rangle$. Then $K$ is a finitely generated $\calA(\bbp)$-module.  Let  $F:=\calA(\bbp)=\langle \varepsilon\rangle$. Then $F$ is a finitely generated free module. 
 Consider the $\calA(\bbp)$-module morphism $\varphi: F\rightarrow I$ given by $\varphi(h)=f_I \ast h$ for all $h\in \calA(\bbp)$. We will show that the sequence 
 $ 0\rightarrow K\hookrightarrow F\rightarrow I \rightarrow 0$ is exact, where $K\hookrightarrow F$ denotes the inclusion map. 
The exactness at $K$ and at $I$ is clear. 
It remains to show 
$ \{h \in \calA(\bbp) : f_I \ast h= 0\} 
= K.
$ 
If $h\in K$, then $h=\chi \ast f$ for an $f\in \calA(\bbp)$. We have $\varphi(h)=f_I\ast( \chi \ast f)$. But 
$\reallywidehat{(f_I\ast \chi)}(k)=0$ for all $k\in \mN_0$, and so $f_I\ast \chi=0$, 
showing that $\varphi(h)=0$, that is, $h\in \ker \varphi$. Hence $K\subset \ker \varphi$. 

Now suppose $h\in \calA(\bbp)$ is such that $f_I\ast h=0$. As $h\in \calA(\bbp)$, 
there exists a $C>0$ such that for all $k\in \mN_0$, 
 $
|\widehat{h}(k)|\leq \frac{C}{\bbp(k)}.$ 
As $f_I\ast h=0$, we have that for all $k\in \mN_0$, 

\vspace{-0.75cm}

\begin{equation}
\label{3320231842}
\begin{array}{c}
0=\reallywidehat{(f_I\ast h)}(k)=\bbp(k) \widehat{f_I}(k) \widehat{h}(k).
\end{array}
\end{equation}
If $k\in \mN_0$ is such that $\widehat{f_I}(k)\neq 0$, then by the definition of $\chi$, we have $\widehat{\chi}(k)=0$, and moreover, then \eqref{3320231842} above implies that $\widehat{h}(k)=0$, so that  
\begin{equation}
\label{3320231846a}
|\widehat{h}(k)|=0=C\cdot 0= C\cdot |\widehat{\chi}(k)|.
\end{equation}
 If $k \in \mN_0$ is such that $\widehat{f_I}(k)=0$, then we have 
$\chi(k)=\frac{1}{\bbp(k)}$, and so 
\begin{equation}
\label{3320231846b}
\begin{array}{c}
|\widehat{h}(k)|\leq \frac{C}{\bbp(k)}=C\cdot \frac{1}{\bbp(k)}=C\cdot |\widehat{\chi}(k)|.
\end{array}
\end{equation}
\eqref{3320231846a} and \eqref{3320231846b} together imply that $|\widehat{h}(k)|\leq C|\widehat{\chi}(k)|$ for all $k\in \mN_0$, and so by the criterion (R1) (p.~\pageref{3320231845}),  $\chi$ divides $h$ in $\calA(\bbp)$, that is, there exists some $f\in \calA(\bbp)$ such that $h=\chi\ast f$, i.e., $h\in \langle \chi\rangle=K$, as wanted. 
\end{proof}

\vspace{-0.6cm}

\section{Generation of $\textrm{SL}_n(\calA(\bbp))$ by elementary matrices}

\vspace{-0.18cm}

\noindent 
Let $R$ be a commutative unital ring with multiplicative identity $1$ and additive identity element $0$. Let $m\in \mN$.  The {\em general linear group} of  invertible matrices  in $R^{m\times m}$ is denoted by $\textrm{GL}_m(R)$. The {\em special linear group} $\SL_m(R)$ is the subgroup of $\textrm{GL}_m(R)$ of 
  all matrices $M$  whose 
  determinant $\det M=1$. An {\em elementary matrix} 
$E_{ij}(\alpha)$ is a matrix having form $E_{ij}=I_m+\alpha \bbe_{ij}$, 
 where $i\neq j$, $\alpha \in R$, and $\bbe_{ij}$ is the $m\times m$ matrix whose entry in the $i^{\textrm{th}}$ row and $j^{\textrm{th}}$  column is $1$, and all the other entries of $\bbe_{ij}$ are zeros. $\E_m(R)$ is the subgroup of $\SL_m(R)$ generated by elementary matrices.  A classical question in algebra is: 
 $$
 \begin{array}{c}
 \textrm{For all }m\in \mN,\textrm{ is }\SL_m(R)=\E_m(R)\textrm{ ?}
 \end{array}
 $$
 
 \noindent  
 The answer to this question depends on the ring $R$. For
example, if the ring $R=\mC$, then the answer is `Yes', and this is
an exercise in linear algebra; see for example \cite[Exercise~18.(c),
page~71]{Art}. If $R$ is the polynomial ring
$\mC[z_1, \cdots, z_d]$ in the indeterminates $z_1, \cdots, z_d$ with
complex coefficients, then for $d=1$, the answer is `Yes' (this
follows from the Euclidean Division Algorithm in $\mC[z]$), but for 
$d=2$, the answer is `No', and \cite{Coh66} contains the
following example:
$$
\left[ \begin{smallmatrix} 1+z_1 z_2 & z_1^2 \\
        -z_2^2 & 1-z_1 z_2 
       \end{smallmatrix}\right] \in \SL_2(\mC[z_1,z_2]) \setminus \E_2(\mC[z_1,z_2]).
$$
(For $d\geq 3$, the answer is `Yes', and this is the $K_1$-analogue
of Serre's Conjecture, which is the Suslin Stability Theorem
\cite{Sus}.)  The case of $R$ being a ring of real/complex valued
continuous functions was considered in \cite{Vas}.  For the ring
$R=\mathcal{O}(X)$ of holomorphic functions on Stein spaces in
$\mC^d$, this was an explicit open
problem  \cite{Gro}, and was answered affirmatively in \cite{IvaKut}. 
We will prove below that
$\SL_n(\calA(\bbp))=\E_n(\calA(\bbp))$.

 For Banach algebras, the following result is known  \cite[\S7]{Mil}:

 \begin{proposition}
  \label{prop_11_may_2021_18:36}  
  Let $A$ be a complex commutative unital semisimple Banach algebra$,$ $n\in \mN,$ and $M\in \eSL_n(A)$. Then $M\in \eE_n (A)$ if and only if 
$M$ is path-connected to $I_n$ in $\eSL_n(A)$ $($i.e.$,$ there exists a continuous map $\gamma:[0,1]\rightarrow \eSL_n(A)$ such that 
$\gamma(0)=M$ and $\gamma(1)=I_n)$. 
 \end{proposition}
 
 \noindent  
 If $(A,\|\cdot\|)$  is a commutative unital Banach algebra, then $A^{n\times n} $ is a complex algebra with the usual matrix operations. 
 Let $A^{n}$ be the normed space of all column vectors of size $n$ with entries from $A$, componentwise operations, and the Euclidean norm given by \eqref{2_3_2023_1449}.  For $M\in A^{n\times n}$,  the multiplication map, 
  $
 A^{n}\owns \bbv \mapsto M\bbv \in A^{n}, 
 $ 
 is 
 a continuous linear transformation, and we equip $A^{n\times n}$ with the induced operator norm, denoted by $\|\cdot\|$ again. Then $A^{n\times n}$ with this operator norm is a unital Banach algebra. Subsets of $A^{n\times n}$ are given the induced  topology.  
 

  We first show some auxiliary results which we will need in the special case when the Banach algebra is $\calA(\bbp)$. Let the operator norm on $M\in \mC^{n\times n}$ (when $\mC^n$ is equipped with the Euclidean norm $\nm\cdot\nm_2$, and $M$ is viewed as a map $\mC^n\owns v\mapsto Mv\in \mC^n$) be denoted by $\nm M\nm_{2,2}$. Let $\calO(\mC)$ denote the set of all entire functions. For a matrix $A\in \calO(\mC)^{n\times n}$, if  $a_{ij}\in \calO(\mC)$ denotes the entry in the $i^{\textrm{th}}$ row and $j^{\textrm{th}}$ column of $A$, and then $\widehat{A}(k)\in \mC^{m\times n}$ is the matrix whose entry in the $i^{\textrm{th}}$ row and $j^{\textrm{th}}$ column is $\widehat{a_{ij}}(k)$, $1\leq i \leq m$, $1\leq j\leq n$, $k\in \mN_0$, where 
  $$
  \begin{array}{c}
  a_{ij}(z)=\sum\limits_{k=0}^\infty \widehat{a_{ij}}(k)z^k\textrm{ for all }z\in \mC.
  \end{array}
  $$ 
  
  \vspace{-0.3cm}
 
 \begin{lemma} 
 \label{9_3_2023_940}
 $A\in \calA(\bbp)^{n\times n}$ if and only if $\sup_{k\in \mN_0} \bbp(k) \nm \widehat{A}(k)\nm_{2,2}<\infty$.  
 Moreover, then $\sup_{k\in \mN_0} \bbp(k) \nm \widehat{A}(k)\nm_{2,2}\leq n\|A\|$. 
 \end{lemma}
 
 \vspace{-0.45cm}
 
 \begin{proof} (`Only if' part:) Suppose $A=[a_{ij}]\in \calA(\bbp)^{n\times n}$. Take the vector $\bbv \in \calA(\bbp)^{n\times 1}$, having only one nonzero entry, which is the $j^{\textrm{th}}$ component, and is equal to $\varepsilon$. Then we obtain 

\vspace{-0.81cm}

 $$
 \begin{array}{c}
 \|a_{ij}\|^2\leq \sum\limits_{i=1}^n \|a_{ij}\|^2 =\|A\bbv\|_2^2\leq \|A\|^2 \|\bbv\|_2^2 =\|A\|^21=\|A\|^2.
 \end{array}
 $$
   Hence $\sup_{k\in \mN_0} \bbp(k)|\widehat{a_{ij}}(k)| = \|a_{ij}\|\leq \|A\|$. As the $i,j$ were arbitrary, it follows from 
 \cite[Fact~9.8.10(xii)]{Ber} that $\sup_{k\in \mN_0} \bbp(k)\nm \widehat{A}(k)\nm_{2,2}\leq n \|A\|$. 
 
 \noindent (`If' part:) If $C:=\sup_{k\in \mN_0} \bbp(k) \nm \widehat{A}(k)\nm_{2,2}<\infty$, then \cite[Fact~9.8.10(xii)]{Ber} implies $\sup_{k\in \mN_0} \bbp(k) |\widehat{a_{ij}}(k)|\leq C$, and so $\|a_{ij}\|\leq C$ for all $1\leq i,j\leq n$. 
 Thus $A\in \calA(\bbp)^{n\times n}$. 
 \end{proof}
 
 \begin{lemma}
 If $A\in \calA(\bbp)^{n\times n}$ then $\nm A(z)\nm_{2,2} \leq n\|A\|\varepsilon(|z|)$ 
 $(z\in \mC)$. 
 \end{lemma}
 
  \vspace{-0.3cm}
  
 \begin{proof}
 For $z\in \mC$, we have 
 
 \vspace{-0.81cm}
 
$$
\phantom{AAA}
\begin{array}{rcl}
\nm A(z)\nm_{2,2} =\nm \sum\limits_{k=0}^\infty \widehat{A}(k) z^k\nm_{2,2} 
\!\!\!&\leq&\!\!\!  \sum\limits_{k=0}^\infty \nm \widehat{A}(k)\nm_{2,2} |z|^k 
\\[0.1cm]
\!\!\!&\leq&\!\!\! \sum\limits_{k=0}^\infty \frac{n\|A\|}{\bbp(k)} |z|^k
=n\|A\|\varepsilon(|z|).\;\!
\phantom{AAA}
\qedhere
\end{array} 
$$
\end{proof}

 \noindent To apply Proposition~\ref{prop_11_may_2021_18:36}, we will need the following result.

\begin{theorem}
\label{8_3_2023_1047}
Let $A\in \eGL_n(\calA(\bbp))$. Then there exists a $B\in \calA(\bbp)^{n\times n}$ such that $e^B=A$.
\end{theorem}

 \vspace{-0.3cm}
 
\noindent Before proving this result, let us see how this gives the following. 

\vspace{-0.12cm}

\begin{theorem}
For all $n\in \mN,$ $\eSL_n(\calA(\bbp))=\eE_n(\calA(\bbp))$. 
\end{theorem}

 \vspace{-0.45cm}
 
\begin{proof} Let $A\in \SL_n(\calA(\bbp))\subset \GL_n(\calA(\bbp))$. Let $B\in \calA(\bbp)^{n\times n}$ be such that $e^B=A$. For $t\in [0,1]$, define $\gamma(t)$ to be the matrix obtained by scaling any one column, say the first one, of $e^{(1-t)B}$, by $\det (e^{-(1-t)B})$. Then $\det (\gamma(t))=\varepsilon$, and so $\gamma(t)\in \SL_n(\calA(\bbp))$ for all $t\in [0,1]$. We have
$$
\gamma(0)=A\textrm{ and }\gamma(1)=e^0=I_n=\left[\begin{smallmatrix} \varepsilon &&\\ & \ddots & \\ && \varepsilon\end{smallmatrix}\right].
$$
 Moreover, as $\gamma$ is continuous, it follows that $A$ is path connected to $I_n$ in $\SL_n(\calA(\bbp))$. From Proposition~\ref{prop_11_may_2021_18:36}, we get $A\in \E_n(\calA(\bbp))$. Thus $\SL_n(\calA(\bbp))=\E_n(\calA(\bbp))$. 
\end{proof}

\noindent Recall that $A\in \GL_n(\mC)$ possesses a logarithm, which can be obtained as follows (see e.g. \cite[Example~5.20]{Kat}). We will be a bit more particular about the construction of the logarithm, since we will apply this to each  $\widehat{A}(k)$, $k\in \mN_0$, for our $A\in \SL_n(\calA(\bbp))$, and we will then need a uniform estimate on $\log (\bbp(k) \widehat{A}(k))$, $k\in \mN_0$. 

Denote the spectrum (the set of eigenvalues) of $A$ by $\sigma(A)$. There exists an open sector $\Omega_\theta=\{z\in \mC\setminus \{0\}: |\textrm{arg}\;\!z-\theta|<\frac{\pi}{n}\}$ of angular width $\frac{2\pi}{n}$ that does not intersect the spectrum of $A$. Moreover,  we have  $0<r:=\min_{\lambda \in \sigma (A)} |\lambda| \leq R:=\max_{\lambda \in \sigma (A)} |\lambda|$. Let $\gamma$ be the path 
$C_{R+1}+S_1+C_{r/2}+S_2$ (see the following picture), where $C_{r/2}$ is a circular arc of radius $r/2$ centred at $0$ traversed in the clockwise direction, $C_{R+1}$ is a circular arc of radius $R+1$ centred at $0$ traversed in the anticlockwise direction,  $S_1$ is a radial straight line segment joining the arc $C_{R+1}$ to the arc $C_{r/2}$ with the fixed argument $\theta-\frac{\pi}{2n}$, and $S_2$ is a radial straight line segment joining the arc $C_{r/2}$ to the arc $C_{R+1}$ with the fixed argument $\theta+\frac{\pi}{2n}$.

\vspace{-0.6cm}

\begin{figure}[H]
      \center
      \psfrag{s}[c][c]{${\scaleobj{0.69}{S_1}}$}
      \psfrag{S}[c][c]{${\scaleobj{0.69}{S_2}}$}
      \psfrag{c}[c][c]{${\scaleobj{0.69}{C_{r/2}}}$}
      \psfrag{C}[c][c]{${\scaleobj{0.69}{C_{R+1}}}$}
      \psfrag{L}[c][c]{${\scaleobj{0.69}{\sigma(A)}}$}
      \psfrag{T}[c][c]{${\scaleobj{0.69}{
      \begin{array}{ll}\textrm{The spectrum $\sigma(A)$  is contained in the shaded region.}\\\textrm{The curve }\gamma=C_{R+1}+S_1+C_{r/2}+S_2\textrm{ encloses }\sigma(A).\end{array}}}$}
      \includegraphics[width=5.4cm]{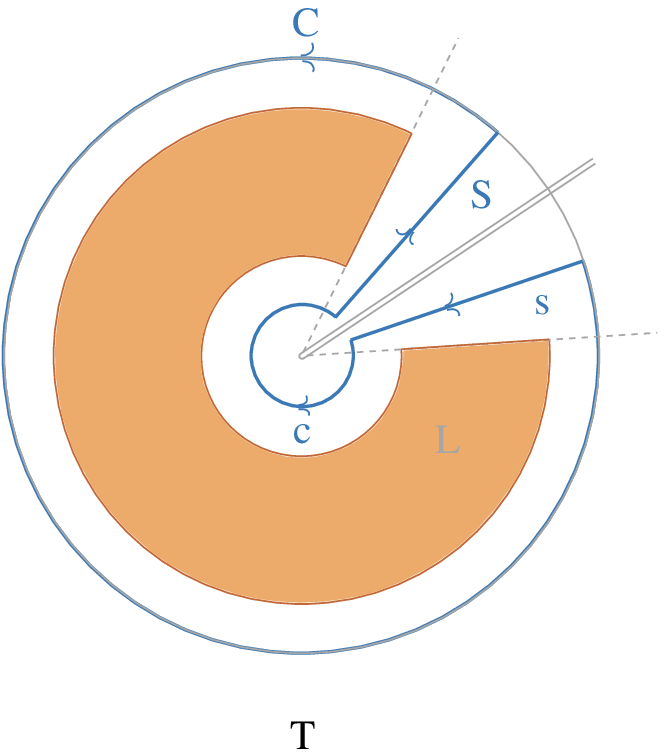}
       
\end{figure}

\noindent  If $\textrm{log}$ denotes the logarithm branch with a cut along the radial ray with fixed argument $\theta$, then we have 
$$
\begin{array}{c}
\log A=\frac{1}{2\pi i}\int_\gamma (\log \zeta)(\zeta I_n-A)^{-1} d\zeta.
\end{array}
$$
We will also use the following estimate\footnote{This follows by setting $p=2$ in the estimate given in \cite[Theorem~4.1]{Ban}, and noting that $\nu_2(A)^2$ given there is bounded by $\|A\|_{\textrm{F}}^2-\sum_{k=1}^n|\lambda_k|^2$, where $\|A\|_{\textrm{F}}$ denotes the Hilbert-Schmidt/Frobenius norm of $A$ and $\lambda_1,\cdots, 
\lambda_n$ denote the $n$ eigenvalues of $A$ repeated with multiplicities. We have $\|A\|_{\textrm{F}}\leq \sqrt{n} \nm A\nm_{2,2}$ (see e.g. \cite[Fact~9.8.10(ix) and Prop. 9.4.7]{Ber}) and $|\lambda_k|\leq \nm A\nm_{2,2}$ for all $k$ \cite[Corollary~9.4.5]{Ber}.} \cite[Theorem~4.1]{Ban} for the norm of the resolvent of $A\in \GL_n(\mC)$ at $z\in \mC\setminus \sigma(A)$: 
\begin{equation}
\label{resolvent_estimate}
\nm (zI_n-A)^{-1}\nm_{2,2} 
\leq \frac{1}{d(z,\sigma(A))} \exp \Big( c_2 \frac{2n \nm A\nm_{2,2}^2}{(d(z,\sigma(A)))^2} +b_2\Big)
\end{equation}
for some universal (not depending on $A$ or $z$) constants $c_2,b_2>0$. Here $d(z,\sigma(A)):=\inf\{|z-\lambda|:\lambda\in \sigma(A)\}$.

 \begin{proof}[Proof of Theorem~\ref{8_3_2023_1047}]  
 Let $A\in \GL_n(\calA(\bbp))$. Then 
 $$
 \begin{array}{c}
 \bbp(k)\widehat{A^{-1}}(k)\widehat{A}(k)=\left[\begin{smallmatrix} \frac{1}{\bbp(k)} &&\\ & \ddots & \\ && \frac{1}{\bbp(k)}\end{smallmatrix}\right].
 \end{array}
 $$
 If $v\in \mC^n\setminus\{0\}$ is an eigenvector of $\widehat{A}(k)$ of unit norm corresponding to the eigenvalue $\widetilde{\lambda}$, and $C>0$ is such that $\nm \widehat{A^{-1}}(k) \nm_{2,2}\leq \frac{C}{\bbp(k)}$, then the above yields upon operation on $v$ that 
 $$
 \begin{array}{c}
 \bbp(k)\frac{C}{\bbp(k)} |\widetilde{\lambda}|\geq \nm \bbp(k)\widehat{A^{-1}}(k)(\widetilde{\lambda} v)\nm_2 
 =\nm \frac{1}{\bbp(k)} v\nm_2=\frac{1}{\bbp(k)},
 \end{array}
 $$
 so that $\min\{|\lambda|:\lambda \in  \sigma(\bbp(k)\widehat{A}(k))\}\geq \frac{1}{C}=:r>0$ for all $k\in\mN_0$. Also, if $R>0$ is such that 
 $\nm \widehat{A}(k) \nm_{2,2}\leq \frac{R}{\bbp(k)}$, then since the spectral radius is bounded by the operator norm,  
 $$
 \begin{array}{c}
 \max\{|\lambda|:\lambda \in  \sigma(\bbp(k)\widehat{A}(k))\}\leq \nm \bbp(k)\widehat{A}(k)\nm_{2,2} \leq R \textrm{ for all } k\in \mN_0.
 \end{array}
 $$
 Let $\Omega_k$ denote an open sector of angular width $\frac{2\pi}{n}$ that does not intersect the spectrum of $\bbp(k)\widehat{A}(k)$. 
  Since $r, R$ do not depend on $k\in \mN_0$, and since the angular wedge width (of $\frac{2\pi}{n}$) also does not depend on $k\in \mN_0$, it is now clear that for any $\zeta$ lying on the image of $\gamma=C_{R+1}+S_1+C_{r/2}+S_2$ (as in the picture above), we have that $|\log_{(k)} \zeta|\leq \widetilde{C}$ for some constant independent of $k\in \mN_0$. (Here we use the notation $\log_{(k)} $ to emphasise the dependence of the chosen branch of the logarithm on the $k$ at hand.) Also the length of $\gamma$ can be bounded by 
  $$
  \begin{array}{c}
  L:=2\pi \frac{r}{2}+2\pi (R+1)+2((R+1)-\frac{r}{2}).
  \end{array}
  $$
   We have 
$$
 \begin{array}{rcl}
 \nm \log_{(k)} (\bbp(k) \widehat{A}(k))\nm_{2,2}
 \!\!\!&=&\!\!\!
 \nm \frac{1}{2\pi i}\int_\gamma (\log \zeta)(\zeta I_n-\bbp(k)\widehat{A}(k))^{-1} d\zeta\nm_{2,2}
 \\[0.2cm]
 \!\!\!&\leq &\!\!\! \frac{L}{2\pi} \widetilde{C} \max_{\zeta \in \gamma} \nm (\zeta I_n-\bbp(k)\widehat{A}(k))^{-1}\nm_{2,2}.
 \end{array}
$$
To bound the final right-hand term involving the resolvent, we will use the estimate \eqref{resolvent_estimate}. First we note that if $\zeta $ lies on $\gamma$, then\footnote{This lower bound is obtained by dropping a perpendicular from the corner of the shaded region onto $S_1$, which has a length $r\sin \frac{\pi}{2n}$, and the distance between $C_{r/2}$ and the shaded region is clearly $r/2$.} 
$$
\begin{array}{c}
d(\zeta, \sigma(\bbp(k)\widehat{A}(k)))\geq \min \{r\sin \frac{\pi}{4n}, \frac{r}{2}\}=:\delta>0.
\end{array}
$$ 
  Also, $\bbp(k) \nm \widehat{A}(k)\nm_{2,2}\leq R$ for all $k\in \mN_0$. Thus 
$$
\begin{array}{c}
\max_{\zeta \in \gamma} \nm (\zeta-\bbp(k)\widehat{A}(k))^{-1}\nm_{2,2}
\leq \frac{1}{\delta}\exp (c_2 \frac{2nR^2}{\delta^2}+b_2)=:K.
\end{array}
$$
This yields $\nm \log_{(k)} (\bbp(k) \widehat{A}(k))\nm_{2,2}\leq \frac{L}{2\pi} \widetilde{C} K=:\widetilde{K}$ for all $k\in \mN_0$. Now define 
$\widehat{B}(k)=\frac{1}{\bbp(k)}\log_{(k)} (\bbp(k) \widehat{A}(k))\in \mC^{n\times n}$ for all $k\in \mN_0$. Then 
$$
\begin{array}{c}
\nm \bbp(k)\widehat{B}(k)\nm_{2,2} \leq \widetilde{K}\textrm{ for all }k\in \mN_0, 
\end{array}
$$
and so by Lemma~\ref{9_3_2023_940}
$$
\begin{array}{c}
B(z):=\sum\limits_{k=0}^\infty \widehat{B}(k)z^k \quad (z\in \mC)
\end{array}
$$
 is an element in $\calA(\bbp)^{n\times n}$. We have 
 $$
\begin{array}{c}B^m(z)=\sum\limits_{k=0}^\infty \bbp(k)^{m-1} (\widehat{B}(k))^m z^k,
\end{array}
$$
 and 
thus\footnote{The exchange of the two summations is justified, 
since: 

$
{\scaleobj{0.96}{\sum\limits_{m=0}^\infty \sum\limits_{k=0}^\infty 
\!\nm \frac{\bbp(k)^{m-1} (\widehat{B}(k))^m}{m!}z^k\nm_{{\scaleobj{0.6}{2,2}}}
\!\!\leq\! 
\sum\limits_{m=0}^\infty \sum\limits_{k=0}^\infty 
\frac{|z|^k}{\bbp(k)} \frac{(n\| B\|)^m}{m!}
\!=\!\sum\limits_{k=0}^\infty \frac{|z|^k}{\bbp(k)} e^{n\|B\|} 
\!=\!e^{n\|B\|} \varepsilon(|z|)\!<\!\infty.}}
$}
$$
\begin{array}{rcl}
(e^B)(z)\!\!\!&=&\!\!\!\sum\limits_{m=0}^\infty \frac{(B^m)(z)}{m!} 
=\sum\limits_{m=0}^\infty\sum\limits_{k=0}^\infty 
\frac{\bbp(k)^{m-1}(\widehat{B}(k))^m }{m!} z^k
\\
\!\!\!&=&\!\!\!
\sum\limits_{k=0}^\infty \frac{z^k}{\bbp(k)} 
\sum\limits_{m=0}^\infty \frac{(\bbp(k)\widehat{B}(k))^m}{m!}
=\sum\limits_{k=0}^\infty \frac{z^k}{\bbp(k)}  e^{\bbp(k) \widehat{B}(k)}
\\
\!\!\!&=&\!\!\!\sum\limits_{k=0}^\infty \frac{z^k}{\bbp(k)} \bbp(k) \widehat{A}(k)
=A(z),
\end{array}
$$
as wanted.
\end{proof}

\end{document}